\documentclass[11pt]{article}
\usepackage[top=2.54cm, bottom=2.54cm, left=2.5cm, right=2.5cm]{geometry}

\usepackage[T1]{fontenc}
\usepackage[utf8]{inputenc}
\usepackage{lmodern}
\usepackage{microtype}

\usepackage[tbtags]{mathtools}
\usepackage{amssymb, latexsym}
\usepackage{amsthm}
\usepackage{mathrsfs}

\usepackage{enumitem}
\usepackage{aliascnt}

\usepackage{xcolor}
\usepackage[numbers,sort&compress]{natbib}
\usepackage[colorlinks=true, linkcolor=blue, citecolor=red, urlcolor=magenta]{hyperref}
\usepackage[nameinlink,capitalize]{cleveref}

\usepackage{lineno}

\usepackage{graphicx}
\usepackage{appendix}
\usepackage{fancyhdr}

\usepackage{float}
\usepackage{pict2e}
\usepackage{tikz}
\usepackage{tikz-3dplot}
\usepackage{tkz-graph}
\usetikzlibrary{arrows,shapes,chains}
\usepackage {verbatim}
\usepackage{scalerel}

\allowdisplaybreaks[4]
\numberwithin{equation}{section}

\theoremstyle{plain}
\newtheorem{theorem}{Theorem}[section]

\newaliascnt{lemma}{theorem}
\newtheorem{lemma}[lemma]{Lemma}
\aliascntresetthe{lemma}

\newaliascnt{corollary}{theorem}
\newtheorem{corollary}[corollary]{Corollary}
\aliascntresetthe{corollary}

\theoremstyle{remark}

\newaliascnt{remark}{theorem}

\aliascntresetthe{remark}

\theoremstyle{definition}

\newaliascnt{definition}{theorem}
\newtheorem{definition}[definition]{Definition}
\aliascntresetthe{definition}

\newaliascnt{conjecture}{theorem}
\newtheorem{conjecture}[conjecture]{Conjecture}
\aliascntresetthe{conjecture}

\newaliascnt{problem}{theorem}

\aliascntresetthe{problem}

\newaliascnt{question}{theorem}
\newtheorem{question}[question]{Question}
\aliascntresetthe{question}

\newaliascnt{claim}{theorem}

\aliascntresetthe{claim}

\newaliascnt{observation}{theorem}

\aliascntresetthe{observation}

\newaliascnt{fact}{theorem}

\aliascntresetthe{fact}

\crefname{theorem}{Theorem}{Theorems}
\crefname{lemma}{Lemma}{Lemmas}
\crefname{corollary}{Corollary}{Corollaries}

\crefname{remark}{Remark}{Remarks}
\crefname{problem}{Problem}{Problems}
\crefname{question}{Question}{Questions}
\crefname{conjecture}{Conjecture}{Conjectures}

\crefname{definition}{Definition}{Definitions}
\crefname{claim}{Claim}{Claims}
\crefname{observation}{Observation}{Observations}
\crefname{fact}{Fact}{Facts}

\usepackage{nomencl}
\makenomenclature

\newcommand{\eps}{\varepsilon}
\newcommand{\cP}{\mathcal P}
\newcommand{\cQ}{\mathcal Q}
\newcommand{\cR}{\mathcal R}
\newcommand{\F}{\mathcal F}
\newcommand{\G}{\mathcal G}
\newcommand{\T}{\mathcal T}
\newcommand{\V}{\mathcal V}
\newcommand{\E}{\mathbb E}
\newcommand{\ex}{\operatorname{ex}}
\newcommand{\Forb}{\operatorname{Forb}}
\newcommand{\gammakl}{\gamma_\ell^{(k)}}

\title{The $\ell$-degree Tur\'an densities of hypergraphs: infinite families and their finite subfamilies}
\author{Hao Lin \thanks{School of Mathematical Sciences, University of Science and Technology of China, Hefei, Anhui 230026, China. Email: \texttt{haolinz6@ustc.edu.cn}.}
\and
Wenling Zhou \thanks{School of Mathematics, Shandong University, Jinan, Shandong 250100, China. Email: \texttt{gracezhou@sdu.edu.cn}.}}
\date{}

\begin{document}
\maketitle

\begin{abstract}
Let $k>\ell\ge 1$ be integers.
For a family $\mathcal{F}$ of $k$-uniform hypergraphs ($k$-graphs), its \emph{$\ell$-degree Tur\'an density} $\gamma^{(k)}_\ell(\mathcal{F})$ is the asymptotic maximum of the normalized minimum $\ell$-degree over all $\mathcal{F}$-free $k$-graphs.
In this paper, we prove that for all integers $k>\ell>2$ and every $\alpha\in[0,1)$, there exists an infinite family $\mathcal{F}$ of $k$-graphs such that $\gamma^{(k)}_\ell(\mathcal{F})=\alpha$, while every finite subfamily $\mathcal{F}'\subset \mathcal{F}$ satisfies $\gamma^{(k)}_\ell(\mathcal{F}')>\alpha$. This confirms the conjecture of Mubayi and Zhao on codegree Tur\'an densities for every $k > 3$. 
As an immediate consequence, the set $\Gamma^{k,\ell}_{\infty}$ of $\ell$-degree Tur\'an densities of nonempty $k$-graph families is $[0,1)$ whenever $k>\ell>2$, which answers a question of Lo and Markstr\"om from 2014.

\end{abstract}

\section{Introduction}\label{sec:introduction}

Tur\'an-type problems ask how dense a large graph or hypergraph can be
while avoiding prescribed subgraphs.  For an integer $k\ge2$, a
\emph{$k$-uniform hypergraph}, or simply a \emph{$k$-graph}, is a pair
$H=(V(H),E(H))$ with $E(H)\subseteq\binom{V(H)}k$.  We write
$v(H)=|V(H)|$ and $e(H)=|E(H)|$.  
Given a nonempty family $\F$ of
$k$-graphs, each containing at least one edge, a $k$-graph is
\emph{$\F$-free} if it contains no member of $\F$ as a subgraph.  The
classical Tur\'an number is
\[
 \ex(n,\F)=\max\{e(H):v(H)=n,\ H\text{ is }\F\text{-free}\},
\]
and the Tur\'an density of $\F$ is
\[
\pi(\F)=\lim_{n\to\infty}\frac{\ex(n,\F)}{\binom nk}.
\]
The limit exists because the normalized Tur\'an numbers form a
non-increasing sequence.  If $\F=\{F\}$, we write $\ex(n,F)$ and
$\pi(F)$.

Let $\Pi^{(k)}$, $\Pi_{\mathrm{fin}}^{(k)}$, and
$\Pi_{\infty}^{(k)}$ denote the sets of Tur\'an densities arising from,
respectively, a single forbidden $k$-graph, a nonempty finite forbidden
family, and an arbitrary nonempty forbidden family.  Thus,
\[
\Pi^{(k)}\subseteq\Pi_{\mathrm{fin}}^{(k)}
\subseteq\Pi_{\infty}^{(k)}\subseteq[0,1).
\]
For graphs, the Erd\H{o}s-Stone-Simonovits theorem~\cite{ErdosStone1946,ErdosSimonovits1966} determines the Tur\'an density of every graph in terms of its chromatic number, and consequently
\begin{equation}\label{eq:graph-spectrum}
\Pi^{(2)}=\Pi_{\mathrm{fin}}^{(2)}=\Pi_{\infty}^{(2)}
 =\{(r-1)/r: r\in \mathbb N^+\}. 
\end{equation}

However, for $k\ge 3$, we do not have an analog of the Erd\H{o}s-Stone-Simonovits theorem, and much less is known about the Tur\'an densities of $k$-graphs.
Tur\'an in his original work~\cite{Tu-graph} asked for the Tur\'an number of complete $k$-graphs. Later, Erd\H{o}s offered \$1000 for determining $\pi(K_n^{(k)})$ for a complete $k$-graph $K_n^{(k)}$ with $n> k\ge3$. However, until now, no case of this problem with $n> k\ge 3$ has been resolved, even asymptotically, despite substantial effort.
For further results along this line of research,
we refer to the surveys of Keevash~\cite{Keevash2011} and
Sidorenko~\cite{Sidorenko1995}.

A different line of research investigates properties of the set of hypergraph Tur\'an densities.
In 1984, Frankl and R\"odl~\cite{FranklRodl1984} disproved a \$1000 conjecture of Erd\H{o}s by showing that, for $k\ge3$, the set $\Pi^{(k)}_\infty$ is not well-ordered. 
Later,
Pikhurko~\cite{Pikhurko2014} established several fundamental properties
of these spectra; in particular, his results imply that
$\Pi_{\infty}^{(k)}$ is uncountable and strictly contains
$\Pi_{\mathrm{fin}}^{(k)}$.

\subsection{\texorpdfstring{$\ell$-degree}{ell-degree} Tur\'an densities}

A natural local refinement of the classical Tur\'an problem replaces
the total number of edges by a minimum $\ell$-degree condition.  Let
$1\le\ell<k$, let $H$ be a $k$-graph, and let
$S\in\binom{V(H)}\ell$.  The \emph{neighborhood} of $S$ is
\[
 N_H(S)=\left\{T\in\binom{V(H)\setminus S}{k-\ell}:
S\cup T\in E(H)\right\},
\]
its \emph{degree} is $d_H(S)=|N_H(S)|$, and the minimum $\ell$-degree of $H$
is $\delta_\ell(H)=\min\{ d_H(S): S\in\binom{V(H)}\ell \}$.
For a family $\F$ of $k$-graphs, define
\[
 \ex_\ell^{(k)}(n,\F)=
 \max\{\delta_\ell(H):v(H)=n,\ H\text{ is }\F\text{-free}\}.
\]
The corresponding \emph{$\ell$-degree Tur\'an density} is
\begin{equation}\label{eq:density}
 \gammakl(\F)=
 \lim_{n\to\infty}
 \frac{\ex_\ell^{(k)}(n,\F)}{\binom{n-\ell}{k-\ell}}.
\end{equation}
The limit exists for arbitrary forbidden families: Mubayi and
Zhao~\cite{MubayiZhao2007} proved this in the codegree case
$\ell=k-1$, and Lo and Markstr\"om~\cite{LoMarkstrom2014} proved it
for all $1\le\ell<k$.  We use the convention
$\gamma_\ell^{(k)}(\varnothing)=1$.  A standard averaging argument gives
\[
 \gamma_{k-1}^{(k)}(\F)\le\gamma_{k-2}^{(k)}(\F)
 \le\cdots\le\gamma_1^{(k)}(\F)=\pi(\F).
\]

As in the classical setting, let $\Gamma^{k,\ell}$,
$\Gamma_{\mathrm{fin}}^{k,\ell}$, and
$\Gamma_{\infty}^{k,\ell}$ denote the sets of $\ell$-degree Tur\'an
densities arising from, respectively, a single forbidden $k$-graph, a
nonempty finite forbidden family, and an arbitrary nonempty forbidden
family. Thus,
\begin{equation}\label{eq:spectrum}
 \Gamma^{k,\ell}\subseteq\Gamma^{k,\ell}_{\mathrm{fin}}
 \subseteq\Gamma^{k,\ell}_{\infty}\subseteq[0,1).
\end{equation}
The first two sets are countable, whereas the third may be uncountable. When $k=2$ (and hence $\ell=1$), these three sets coincide with the
graph spectrum in \eqref{eq:graph-spectrum}. By contrast, Mubayi and Zhao~\cite{MubayiZhao2007} proved that $\Gamma^{k,k-1}_\infty$ is dense in $[0,1)$ for every $k\ge3$  and posed the following conjecture.
\begin{conjecture}[{\cite[Conjecture~1.7]{MubayiZhao2007}}]
	\label{conj:zhao}
	For every integer $k\ge3$ and every $\alpha\in[0,1)$, there exists an
	infinite family $\F$ of $k$-graphs such that $\gamma_{k-1}^{(k)}(\F)=\alpha$
	and $\gamma_{k-1}^{(k)}(\F')>\alpha$
	for every finite subfamily $\F'\subseteq\F$.
\end{conjecture}

Subsequently, Lo and Markstr\"om~\cite[Theorem~1.7]{LoMarkstrom2014}
proved that $\Gamma_{\infty}^{k,\ell}$ is dense in $[0,1)$ for every
$k>\ell>1$.  They further posed the following natural question.

\begin{question}[Lo and Markstr\"om~\cite{LoMarkstrom2014}]
	\label{ques:lo-markstrom}
	Is it true that $\Gamma_{\infty}^{k,\ell}=[0,1)$
	for every pair of integers $k>\ell>1$?
\end{question}

There has also been substantial recent progress for finite forbidden
families.  Gao, Pikhurko, Rong, and Sun
\cite{GaoPikhurkoRongSun2026} proved that every rational number in
$[0,1)$ is the codegree Tur\'an density of a finite family of
$k$-graphs.  Tang and the second author~\cite{TangZhou2026} extended this
rational-value result to every pair $k>\ell>k/2$.  These results yield
many values in $\Gamma_{\mathrm{fin}}^{k,\ell}$, including all rational
values in the ranges above.  Nevertheless,
$\Gamma_{\mathrm{fin}}^{k,\ell}$ is countable, so such results cannot by
themselves show that an entire interval is contained in
$\Gamma_{\infty}^{k,\ell}$.

\subsection{Main result and proof overview}
A finite approximation result of Lo and Markstr\"om
\cite[Corollary~2.5]{LoMarkstrom2014} shows that, for every family $\mathcal F$
of $k$-graphs,
\[
\gamma_\ell^{(k)}(\mathcal F)
=
\inf\bigl\{
\gamma_\ell^{(k)}(\mathcal F'):
\mathcal F'\subseteq\mathcal F,\ 
\mathcal F'\text{ is finite}
\bigr\}.
\]
Thus, the density of an arbitrary forbidden family can be
approximated arbitrarily closely by the densities of its finite
subfamilies.

The natural $\ell$-degree analog of \cref{conj:zhao} asks whether,
for every pair $k>\ell>1$ and every $\alpha\in[0,1)$, there exists an
infinite family $\F$ of $k$-graphs such that $\gamma_\ell^{(k)}(\F)=\alpha
<\gamma_\ell^{(k)}(\F')$
for every finite subfamily $\F'\subseteq\F$.  Our main theorem proves
this statement whenever $\ell>2$.

\begin{theorem}\label{thm:main}
	For every pair of integers $k>\ell>2$ and every $\alpha\in[0,1)$,
	there exists an infinite family $\mathcal F$ of $k$-graphs such that
	\[
	\gamma_\ell^{(k)}(\mathcal F)
	= \alpha
	< \gamma_\ell^{(k)}(\mathcal F')
	\]
	for every finite subfamily $\mathcal F'\subseteq\mathcal F$.
\end{theorem}

For $\ell=k-1$, \cref{thm:main} proves \cref{conj:zhao} for every
$k>3$.  
By retaining only the density conclusion, we also obtain the
following complete description of the spectrum and answer  \cref{ques:lo-markstrom} whenever $k>\ell>2$.

\begin{corollary}\label{cor:spectrum}
For all integers $k>\ell>2$, one has
$\Gamma^{k,\ell}_{\infty}=[0,1)$.
\end{corollary}

During the preparation of this manuscript, we learned of independent
work of Li and Liu~\cite{Li-Liu-26}, who proved that $\Gamma_{\infty}^{k,\ell}=[0,1)$
for all integers $k>\ell>1$ with $(k,\ell)\ne(3,2)$.  Their result
therefore answers \cref{ques:lo-markstrom} for a wider range of
parameters.  However, in the common range $k>\ell>2$, our result also ensures that every finite subfamily of the constructed forbidden family has density strictly larger than the prescribed value.

\paragraph{Proof overview.}
The main idea is to separate the behavior of an infinite forbidden
family from that of its finite subfamilies.  Rather than constructing
the forbidden graphs directly, we construct a collection of monotone
decreasing families $(\cQ_i)_{i\in I}$ and take $\cQ:=\bigcap_{i\in I}\cQ_i$.
Every graph outside $\cQ$ is excluded by at least one of the families
$\cQ_i$.  Consequently, every finite subfamily of the full forbidden
family of $\cQ$ is already excluded by the intersection of finitely
many $\cQ_i$.  Thus, the desired conclusion follows if the full
intersection has density $\alpha$, while every finite intersection
has density strictly larger than~$\alpha$.  

To construct monotone families with controllable densities, we use a
cyclic multipartite construction.  The vertex set is divided into
cyclically ordered parts.  The construction automatically gives many
extensions to $\ell$-sets meeting more than one part, while the
extensions of an $\ell$-set lying inside one part are governed by an
auxiliary $\ell$-graph $G$.  This separates the contribution of the
partition from that of the auxiliary graph and yields the density
formula $c_{m,t}\lambda(G)$,
or, when $G$ ranges over a family $\G$, $c_{m,t}\sup_{G\in\G}\lambda(G)$.
In this way, Lagrangian values of auxiliary $\ell$-graphs become
$\ell$-degree Tur\'an densities of monotone families of $k$-graphs.

For $\alpha>0$, we choose the cyclic parameters and a Lagrangian
threshold $\tau$ so that $\alpha=c_{m,t}\tau$.  Allowing all auxiliary
graphs with Lagrangian below $u$ gives nested monotone families whose
densities are $c_{m,t}u$.  Taking thresholds $u_i\downarrow\tau$
therefore produces a full intersection of density $\alpha$, whereas
every finite intersection is one of the earlier families and has
density strictly larger than $\alpha$.

For $\alpha=0$, we instead use cyclic multipartite families
$\cR_m$ whose densities tend to zero as $m\to\infty$.  Their full
intersection consequently has density zero.  The essential additional
step is to show that every finite intersection of the families
$\cR_m$ still has positive density.  This is achieved by refining the
corresponding vertex partitions simultaneously, so that the
requirements of all the finitely many constructions can be satisfied
at once.

\paragraph{Organization.}
The remainder of the paper is organized as follows. In \cref{sec:preliminaries}, we collect the notations and auxiliary
tools used later.  In \cref{sec:construction}, we
introduce the cyclic multipartite families and determine their
$\ell$-degree Tur\'an densities.  In \cref{sec:main-proof}, we prove
\cref{thm:main}. Finally,
in \cref{sec:conclusion}, we summarize the results and discuss further questions.

\section{Preliminaries}\label{sec:preliminaries}
In this section, we present the notations and auxiliary results used in the paper.
For a finite set $X$, let $\binom X k$ denote the family of its
$k$-element subsets. For a positive integer $n$, let
$[n]=\{1,\ldots,n\}$.
The complete $k$-graph on $n$ vertices is denoted by $K_n^{(k)}$.

We use Hoeffding's inequality~\cite{Hoeffding1963} in the following
form. If $X_1,\ldots,X_n$ are independent random variables with
values in $[0,1]$, then, for every $u>0$,
\begin{equation}\label{eq:hoeffding}
	\Pr\left[\sum_{j=1}^n X_j
	<\E\sum_{j=1}^n X_j-u\right]
	\le \exp\left(-\frac{2u^2}{n}\right).
\end{equation}

\subsection{Monotone families and forbidden $k$-graphs}
\label{sec:monotone-families}
In this subsection, we recall monotone families and their description
by forbidden subgraphs.
\begin{definition}
	A family $\cP$ of $k$-graphs is \emph{monotone} if it is closed under taking subgraphs.
\end{definition}

It follows immediately that arbitrary unions and intersections of
monotone families are again monotone.
We now give the definition of \emph{forbidden families}, which provide
the basic examples of monotone families.
\begin{definition}\label{def:forbidden-family}
	For a family $\F$ of $k$-graphs, let
	\[
	\Forb(\F)
	:=\{H:H\text{ contains no member of $\F$ as a subgraph}\}.
	\]
\end{definition}

The family $\Forb(\F)$ is monotone.  Indeed, if
$H\in\Forb(\F)$ and $H'$ is a subgraph of $H$, then any member of
$\F$ contained in $H'$ would also be contained in $H$.
Conversely, every monotone family
can be described by forbidding all $k$-graphs outside it.

\begin{definition}
\label{def:forbidden-famlily}
For a monotone family $\cP$ of $k$-graphs, define its
\emph{full forbidden family} by
\[
\F(\cP)
:=\{F:F\text{ is a $k$-graph and }F\notin\cP\}.
\]
\end{definition}
The definition immediately gives the following identity.
\begin{lemma}
\label{fact:family-equivalence}
For every monotone family $\cP$ of $k$-graphs, we have $\Forb(\F(\cP))=\cP$.
\end{lemma}
\begin{proof}
If $H\in\cP$,
then every subgraph of $H$ belongs to $\cP$, and hence $H$ contains no
member of $\F(\cP)$.  Thus $H\in\Forb(\F(\cP))$.  On the other hand,
if $H\notin\cP$, then $H\in\F(\cP)$, so $H\notin\Forb(\F(\cP))$.
\end{proof}

Consequently, for a monotone family $\cP$ containing $k$-graphs of
arbitrarily large order, \cref{fact:family-equivalence} gives
\begin{equation}\label{eq:family-density}
	\gammakl(\F(\cP))
	=\lim_{n\to\infty}
	\frac{\max\{\delta_\ell(H):H\in\cP,\ v(H)=n\}}
	{\binom{n-\ell}{k-\ell}}.
\end{equation}
We will repeatedly use the following monotonicity properties.  If
$\F_1\subseteq\F_2$ are forbidden families, then $\gammakl(\F_1)\ge\gammakl(\F_2)$.
If $\cP\subseteq\cQ$ are such monotone families, then
$\gammakl(\F(\cP))\le\gammakl(\F(\cQ))$.
Also, for every collection $(\cP_i)_{i\in I}$ of monotone families, we have $\F\left(\bigcap_{i\in I}\cP_i\right)
=\bigcup_{i\in I}\F(\cP_i)$.

\subsection{Hypergraph Lagrangians and Tur\'an densities}
\label{sec:lagrangians}

We next recall the hypergraph Lagrangian and the results on
classical Tur\'an densities that will be used later.

\begin{definition}\label{def:lagrangian}
For a $k$-graph $G$, its \emph{normalized Lagrangian} is
\begin{equation}\label{eq:lagrangian}
 \lambda(G)
 :=\max\left\{
 k!\sum_{e\in E(G)}\prod_{v\in e}x_v:
 x_v\ge0\text{ for every }v\in V(G),
 \sum_{v\in V(G)}x_v=1
 \right\}.
\end{equation}
\end{definition}

The maximum is attained because the feasible simplex is compact. 
For any
feasible weight vector $\boldsymbol{x}=(x_v)_{v\in V(G)}$, the
quantity $k!\sum_{e\in E(G)}\prod_{v\in e}x_v$
is the probability that $k$ independent samples from $V(G)$ according
to $\boldsymbol{x}$ are distinct and form an edge of $G$.  Thus, $0\le\lambda(G)\le1$.

Let $\Lambda^{(k)}$ be the set of Lagrangians of $k$-graphs. The following two results on
classical Tur\'an densities are used in the proof for positive
densities. The cited papers allow the empty forbidden family and hence include
\(1\) in \(\Pi_\infty^{(k)}\). Since we consider only nonempty
forbidden families, we state their results in the following form.

\begin{theorem}[Pikhurko~\cite{Pikhurko2014}]\label{thm:closure}
For every integer $k\ge3$, $\overline{\Lambda^{(k)}}
=\Pi_\infty^{(k)}\cup\{1\}$.
\end{theorem}

\begin{theorem}[Liu-Pikhurko~\cite{LiuPikhurko2026}]\label{thm:interval-near-one}
For every integer $k\ge3$, there is a constant $\delta_k>0$ such that
\[[1-\delta_k,1)\subseteq\Pi_\infty^{(k)}.\]
\end{theorem}

Combining these two theorems, we can obtain the following consequence.

\begin{lemma}\label{lem:one-sided}
For every integer \(k\ge3\), there exists \(b_k\in(0,1)\) such that,
for every \(\tau\in(b_k,1)\),
\[
\tau=
\sup\{\lambda(G):
G\text{ is a \(k\)-graph and }\lambda(G)<\tau\}.
\]
\end{lemma}

\begin{proof}
Choose $b_k\in(0,1)$ so that $[b_k,1)\subseteq\Pi_\infty^{(k)}$,
using \cref{thm:interval-near-one}.  Fix $\tau\in(b_k,1)$ and
$\eps>0$.  Choose
\[
 \max\{b_k,\tau-\eps\}<\beta<\tau,
\]
and then choose $\eta>0$ so small that
$(\beta-\eta,\beta+\eta)\subseteq(\tau-\eps,\tau)$.  Since
$\beta\in\Pi_\infty^{(k)}$, \cref{thm:closure} gives a finite
$k$-graph $G$ with $|\lambda(G)-\beta|<\eta$.  Thus
$\tau-\eps<\lambda(G)<\tau$, as required.
\end{proof}

\section{Cyclic multipartite families and their densities}\label{sec:construction}
Throughout this section, fix integers $k>\ell \ge 2$. 
Let $m,t$ be
integers satisfying
$m\ge k$ and $k-\ell\le t\le m-\ell$, and set
\[
	c_{m,t}:=\frac{(k-\ell)!\binom{t}{k-\ell}}{m^{k-\ell}}.
\]

The purpose of this section is to convert Lagrangian information about
auxiliary $\ell$-graphs into $\ell$-degree Tur\'an densities of
$k$-graphs.  In \cref{subsec:construction}, for each auxiliary
$\ell$-graph $G$ we define a monotone family $\cP(G;m,t)$ by a cyclic
$m$-partite construction.  In \cref{sec:fixed-auxiliary},
we show that the largest normalized
minimum $\ell$-degree among the $n$-vertex members of $\cP(G;m,t)$
converges to $c_{m,t}\lambda(G)$ for each fixed auxiliary $\ell$-graph $G$.  Finally, in
\cref{sec:auxiliary-hypergraphs}, we allow $G$ to range over an
arbitrary nonempty family and determine the resulting $\ell$-degree
Tur\'an density from the supremum of the corresponding Lagrangians.

All asymptotic notation in this section is with respect to
$n\to\infty$.  Unless stated otherwise, implicit constants may depend
on $k,\ell,m,t$, but are independent of the auxiliary graph $G$ and of
the choices made in the construction.

\subsection{The construction}\label{subsec:construction}
We now give the precise construction.  Its input is an auxiliary
$\ell$-graph $G$.  The vertices of the resulting $k$-graph are divided
into cyclically ordered parts, and $G$ controls the edges containing
$\ell$ vertices in one part.

\begin{definition}\label{def:cyclic-multipartite-construction}
	Let $G$ be an $\ell$-graph, let $V$ be a vertex set, and let
	$\V=(V_0,\ldots,V_{m-1})$ be an ordered partition of $V$ into $m$
	parts, where empty parts are allowed and indices are taken modulo $m$.
	For each $i\in\mathbb Z_m$, put
	$I_i:=\{i+1,\ldots,i+t\}$ and define
	\begin{equation}\label{eq:extensions}
		\T_i:=
		\left\{T\in\binom{V}{k-\ell}:
		T\subseteq\bigcup_{j\in I_i}V_j
		\text{ and }|T\cap V_j|\le1\text{ for every }j\in I_i\right\}.
	\end{equation}
	For each $i\in\mathbb Z_m$ and $T\in\T_i$, let
	$\phi_{i,T}:V_i\to V(G)$ be a map, and write
	$\Phi=(\phi_{i,T})_{i\in\mathbb Z_m,\,T\in\T_i}$.
	Define $H(G;\V,\Phi)$ to be the $k$-graph on $V$ in which
	$X\in\binom Vk$ is an edge if and only if one of the following holds:
	\begin{enumerate}[label=\textup{(\roman*)},leftmargin=*,itemsep=3pt]
		\item $|X\cap V_i|\le \ell-1$ for every $i\in\mathbb Z_m$;
		\item there exist $i\in\mathbb Z_m$, $S\in\binom{V_i}\ell$, and
		$T\in\T_i$ such that
		\[
		X=S\mathbin{\dot\cup}T
		\qquad\text{and}\qquad
		\{\phi_{i,T}(u):u\in S\}\in E(G).
		\]
	\end{enumerate}
\end{definition}

We call a $k$-graph of the form $H(G;\V,\Phi)$ a
\emph{cyclic multipartite $G$-construction}. Note that each member of $\T_i$ contains one vertex from
each of $k-\ell$ distinct parts among
$V_{i+1},\ldots,V_{i+t}$.
Since $\ell\ge2$, an edge of type~\textup{(ii)} contains exactly
$\ell$ vertices in a unique part $V_i$ and at most one vertex in every
other part. Thus the
two edge types are disjoint, and the index $i$ in
\textup{(ii)} is uniquely determined. Moreover,
$\{\phi_{i,T}(u):u\in S\}\in E(G)$ implies that
$\phi_{i,T}$ is injective on $S$.

The dependence of $\phi_{i,T}$ on $T$ is deliberate.  For a fixed pair
$(i,T)$, the $\ell$-graph on $V_i$ whose edges are the sets
$S\in\binom{V_i}\ell$ satisfying
$S \mathbin{\dot\cup} T\in E(H(G;\V,\Phi))$ is a possibly degenerate blow-up of $G$,
with vertices given by the preimages of $\phi_{i,T}$.  This allows
its number of edges to be controlled by $\lambda(G)$.  At the same
time, maps corresponding to different extension sets $T$ may be
chosen independently, which is the key feature used in the lower-bound
construction.

We now give the monotone family generated by $H(G;\V,\Phi)$.

\begin{definition}\label{def:construction-family}
	For an $\ell$-graph $G$, let $\cP(G;m,t)$ be the family of all
	$k$-graphs $H$ for which there exist an ordered $m$-partition
	$\V$ of $V(H)$ and maps $\Phi$ as in
	\cref{def:cyclic-multipartite-construction} such that
	\[
	E(H)\subseteq E\bigl(H(G;\V,\Phi)\bigr).
	\]
\end{definition}

Thus $\cP(G;m,t)$ consists of all $k$-graphs that occur as subgraphs
of cyclic multipartite $G$-constructions.  It follows directly from
the definition that $\cP(G;m,t)$ is monotone.  Indeed, deleting edges
does not change the witnessing partition and maps, while after deleting
vertices we restrict the parts and the corresponding maps.

Moreover, every part of the witnessing partition is independent.  Indeed, an edge of type~\textup{(i)} meets each part in at most $\ell-1$ vertices,
while an edge of type~\textup{(ii)} meets its distinguished part in
exactly $\ell$ vertices and every other part in at most one vertex.
Consequently, every clique in a member of $\cP(G;m,t)$ contains at most
$k-1$ vertices from each part.  Hence every member of
$\cP(G;m,t)$ is $K_{m(k-1)+1}^{(k)}$-free, and therefore $\F\bigl(\cP(G;m,t)\bigr)\ne\varnothing$.

\subsection{The density for a fixed auxiliary graph}
\label{sec:fixed-auxiliary}

We next determine the density of $\cP(G;m,t)$ for a fixed auxiliary
$\ell$-graph $G$.  We first need the following averaging lemma for
cyclically consecutive parts. For an
ordered partition $\V=(V_0,\ldots,V_{m-1})$, the definition of $\T_i$
gives
\[
 |\T_i|
 =\sum_{A\in\binom{[t]}{k-\ell}}
   \prod_{j\in A}|V_{i+j}|.
\]
The following weighted form will allow us to find a nonempty part for
which $|\T_i|$ is no larger than its value in the balanced case.

\begin{lemma}\label{lem:cyclic}
Let $y_0,\ldots,y_{m-1}$ be nonnegative real numbers with
$n:=\sum_{j=0}^{m-1}y_j>0$.  Then there exists $i\in\mathbb Z_m$ with
$y_i>0$ such that
\[
 \sum_{A\in\binom{[t]}{k-\ell}}
 \prod_{j\in A}y_{i+j}
 \le \binom{t}{k-\ell}\left(\frac{n}{m}\right)^{k-\ell}.
\]
\end{lemma}

\begin{proof}
For $j\in\mathbb Z_m$, let
\[
 W_j:=\sum_{h=1}^t y_{j+h}.
\]
Since $\sum_{j\in\mathbb Z_m}W_j=tn$, some $j\in\mathbb Z_m$ satisfies
$W_j\le tn/m$.  If $y_j>0$, take $i=j$.  Otherwise, let $i$ be the first index of positive weight encountered
when moving backward from $j$ in the cyclic order, and let $d$ be the
forward distance from $i$ to $j$.  Then
$y_{i+1}=\cdots=y_{i+d}=0$.  If $d\ge t$, then $W_i=0$.  If $d<t$,
the first $d$ summands in $W_i$ vanish and every remaining summand also
occurs in $W_j$.  Thus, in either case,
\[
 y_i>0
 \qquad\text{and}\qquad
 W_i\le W_j\le \frac{tn}{m}.
\]
Maclaurin's inequality applied to
$y_{i+1},\ldots,y_{i+t}$ now gives
\[
 \sum_{A\in\binom{[t]}{k-\ell}}
 \prod_{j\in A}y_{i+j}
 \le \binom{t}{k-\ell}
 \left(\frac{W_i}{t}\right)^{k-\ell}
 \le \binom{t}{k-\ell}
 \left(\frac{n}{m}\right)^{k-\ell}.
\]
\end{proof}

We can now prove the matching degree estimates. 

\begin{theorem}\label{thm:fixed-auxiliary}
Let $k>\ell\ge2$, and $m,t$ be integers satisfying $m\ge k$, $k-\ell\le t\le m-\ell$.
There exist a constant $C>0$ and a positive integer $n_0$, depending
only on $k,\ell,m,t$, such that, for every $\ell$-graph $G$ and every
$n\ge n_0$, the following hold:
\begin{enumerate}[label=\textup{(\roman*)},leftmargin=*,itemsep=4pt]
 \item Every $n$-vertex $k$-graph $H\in\cP(G;m,t)$ satisfies
 \begin{equation}\label{eq:fixed-auxiliary-upper}
 	\frac{\delta_\ell(H)}{\binom{n-\ell}{k-\ell}}
 \le c_{m,t}\lambda(G)+Cn^{-1/2}.
 \end{equation}
 \item There exists an $n$-vertex $k$-graph $H\in\cP(G;m,t)$ satisfying 
 \begin{equation}
 	\label{eq:fixed-auxiliary-lower} 
 	\frac{\delta_\ell(H)}{\binom{n-\ell}{k-\ell}}
 \ge c_{m,t}\lambda(G)
 -C\left(n^{-1}+n^{-(k-\ell)/2}\sqrt{\log n}\right).
 \end{equation}
\end{enumerate}
Consequently,
\begin{equation}\label{eq:fixed-auxiliary-density}
 \gammakl\bigl(\F(\cP(G;m,t))\bigr)=c_{m,t}\lambda(G).
\end{equation}
\end{theorem}

\begin{proof}
\emph{Upper bound.}
Let $H\in\cP(G;m,t)$ have $n$ vertices.  By
\cref{def:construction-family}, $H$ is a spanning subgraph of some
$H(G;\V,\Phi)$.  Adding edges cannot decrease the minimum
$\ell$-degree, so it is enough to consider the case
$H=H(G;\V,\Phi)$.

Put $n_j:=|V_j|$ and
define
\[
y_j:=
\begin{cases}
	n_j,& n_j\ge\sqrt n,\\
	0,& n_j<\sqrt n.
\end{cases}
\]
If every $y_j$ were zero, then $n<m\sqrt n$, which is impossible for
$n>m^2$. Thus, for all sufficiently large $n$, at least one $y_j$ is
positive. Let $N:=\sum_{j=0}^{m-1}y_j\le n$. Applying
\cref{lem:cyclic} gives an index $i$ with $n_i=y_i\ge\sqrt n$ and
\begin{equation}\label{eq:large-extensions-upper}
	\sum_{A\in\binom{[t]}{k-\ell}}\prod_{j\in A}y_{i+j}
	\le\binom{t}{k-\ell}\left(\frac Nm\right)^{k-\ell}
	\le\binom{t}{k-\ell}\left(\frac nm\right)^{k-\ell}.
\end{equation}

The sum on the left counts exactly those members of $\T_i$ that avoid
all parts of size less than $\sqrt n$. The union of all smaller parts
contains fewer than $m\sqrt n$ vertices. Every member of $\T_i$ not
counted in \eqref{eq:large-extensions-upper} contains at least one such vertex.
Choosing one such vertex first and then choosing the
other $k-\ell-1$ vertices, while ignoring the restrictions in the
definition of $\T_i$, gives
\[
\bigl|\{T\in\T_i:T\text{ meets a part of size less than }\sqrt n\}\bigr|
\le m\sqrt n\binom{n-1}{k-\ell-1}
=O(n^{k-\ell-1/2}).
\]
Consequently,
\begin{equation}\label{eq:number-extensions-upper}
	|\T_i|
	\le\binom{t}{k-\ell}\left(\frac nm\right)^{k-\ell}
	+O(n^{k-\ell-1/2}).
\end{equation}

Fix $T\in\T_i$.  Since the preimages of $\phi_{i,T}$ partition $V_i$,
the vector $(|\phi_{i,T}^{-1}(c)|/n_i)_{c\in V(G)}$
is feasible in the definition of the normalized Lagrangian of $G$.
It follows that
\[
	\bigl|\{S\in\tbinom{V_i}{\ell}:
	S\mathbin{\dot\cup}T\in E(H)\}\bigr| =\sum_{e\in E(G)}
	\prod_{c\in e}|\phi_{i,T}^{-1}(c)| =n_i^\ell\sum_{e\in E(G)}
	\prod_{c\in e}\frac{|\phi_{i,T}^{-1}(c)|}{n_i} \le \frac{\lambda(G)}{\ell!}n_i^\ell.
\]
No edge of type~\textup{(i)} contains an $\ell$-set lying in $V_i$.
Moreover, a type~\textup{(ii)} edge has a unique part containing
$\ell$ vertices.  Therefore, double-counting the pairs $(S,T)$ with
$S\in\binom{V_i}{\ell}$, $T\in\T_i$, and
$S\mathbin{\dot\cup}T\in E(H)$ yields
\[
\begin{aligned}
	\delta_\ell(H) \le \frac{1}{\binom{n_i}{\ell}}
	\sum_{S\in\binom{V_i}{\ell}}d_H(S) =\frac{1}{\binom{n_i}{\ell}}
	\sum_{T\in\T_i}
	\bigl|\{S\in\tbinom{V_i}{\ell}:
	S\mathbin{\dot\cup}T\in E(H)\}\bigr| \le \lambda(G)
	\frac{n_i^\ell}{\ell!\binom{n_i}{\ell}}|\T_i|.
\end{aligned}
\]
Since $n_i\ge\sqrt n$,
\[
\frac{n_i^\ell}{\ell!\binom{n_i}{\ell}}
=\prod_{h=0}^{\ell-1}\left(1-\frac{h}{n_i}\right)^{-1}
=1+O\bigl(n^{-1/2}\bigr).
\]
Together with \eqref{eq:number-extensions-upper} and
$0\le\lambda(G)\le1$, this gives
\[
\delta_\ell(H)
\le \lambda(G)\binom{t}{k-\ell}
\left(\frac{n}{m}\right)^{k-\ell}
+O\bigl(n^{k-\ell-1/2}\bigr).
\]
Moreover,
\[
\binom{n-\ell}{k-\ell}
=
\frac{n^{k-\ell}}{(k-\ell)!}
\bigl(1+O(n^{-1})\bigr).
\]
Hence
\[
	\frac{\delta_\ell(H)}{\binom{n-\ell}{k-\ell}} \le
	\frac{(k-\ell)!\binom{t}{k-\ell}}{m^{k-\ell}}
	\lambda(G)+O(n^{-1/2}) =c_{m,t}\lambda(G)+O(n^{-1/2}).
\]
Since the implicit constant depends only on $k,\ell,m,t$, this proves
\eqref{eq:fixed-auxiliary-upper} with a constant independent of $G$.

\medskip
\noindent\emph{Lower bound.}
Choose an ordered partition
$\V=(V_0,\ldots,V_{m-1})$ satisfying
\[
\left||V_i|-\frac{n}{m}\right|\le1
\qquad\text{for every }i\in\mathbb Z_m.
\] Let
$\boldsymbol{x}=(x_v)_{v\in V(G)}$ be a weight vector attaining
$\lambda(G)$.  Independently for every
$i\in\mathbb Z_m$, $T\in\T_i$, and $u\in V_i$, choose the image
$\phi_{i,T}(u)$ from $V(G)$ according to the distribution
$\boldsymbol{x}$, and let $H:=H(G;\V,\Phi)$.

For every $i\in\mathbb Z_m$, the balanced partition gives
\[ n^{k-\ell} \ge |\T_i| \ge
	\binom{t}{k-\ell}
	\left(\frac{n}{m}-1\right)^{k-\ell} \ge
	\binom{t}{k-\ell}
	\left(\frac{n}{m}\right)^{k-\ell} - O\bigl(n^{k-\ell-1}\bigr).
\]

First, consider an $\ell$-set $S$ lying in one part, i.e. $S\in\binom{V_i}{\ell}$ for some $i\in\mathbb Z_m$.  For each
$T\in\T_i$, the probability that
$\{\phi_{i,T}(u):u\in S\}$ is an edge of $G$ is $\lambda(G)$.
The corresponding events are independent for distinct $T$, because
the random choices defining $\phi_{i,T}$ are indexed by $T$.  Hence
\[
 d_H(S)\sim\operatorname{Bin}(|\T_i|,\lambda(G)).
\]
Thus, Hoeffding's inequality
\eqref{eq:hoeffding} gives
\[
 \Pr\bigl[d_H(S)<\lambda(G)|\T_i|- n^{(k-\ell)/2}\sqrt{(\ell+1) \log n } \bigr]
 \le \exp\left(-\frac{2n^{k-\ell}(\ell+1)\log n}{|\T_i|}\right)
 \le n^{-2\ell-2}.
\]
There are at most $n^\ell$ choices for an $\ell$-set contained in one
part.  A union bound therefore shows that, with positive probability,
every such set $S$ satisfies
\begin{equation}\label{eq:within-part-lower}
 d_H(S)
 \ge \lambda(G)\binom{t}{k-\ell}
       \left(\frac{n}{m}\right)^{k-\ell}
    -O\left(n^{k-\ell-1}
       +n^{(k-\ell)/2}\sqrt{\log n}\right).
\end{equation}

It remains to consider an $\ell$-set $S$ meeting at least two parts.  Let
$J:=\{j\in\mathbb Z_m:S\cap V_j\ne\varnothing\}$, so $2 \le |J|\le\ell$.
Every choice of $k-\ell$ parts outside $J$, followed by one vertex from
each chosen part, produces a set $T$ for which
$S\mathbin{\dot\cup}T$ is an edge of type~\textup{(i)}.  
Since every part has size at least $n/m-1$, we obtain
\[
	d_H(S) \ge
	\binom{m-|J|}{k-\ell}
	\left(\frac{n}{m}-1\right)^{k-\ell} \ge
	\binom{t}{k-\ell}
	\left(\frac{n}{m}\right)^{k-\ell}
	- O(n^{k-\ell-1})
\]
As
$\lambda(G)\le1$, this bound is at least the right-hand side of
\eqref{eq:within-part-lower}, after adjusting the implicit constant.

Choose a hypergraph for which \eqref{eq:within-part-lower} holds for all
$\ell$-sets contained in one part.  The preceding two estimates then
give
\[
 \delta_\ell(H)
 \ge \lambda(G)\binom{t}{k-\ell}
       \left(\frac{n}{m}\right)^{k-\ell}
    -O\left(n^{k-\ell-1}
       +n^{(k-\ell)/2}\sqrt{\log n}\right).
\]
Dividing by $\binom{n-\ell}{k-\ell}$ yields
\[
 \frac{\delta_\ell(H)}{\binom{n-\ell}{k-\ell}}
 \ge c_{m,t}\lambda(G)
    -O\left(n^{-1}
       +n^{-(k-\ell)/2}\sqrt{\log n}\right),
\]
which proves ~\eqref{eq:fixed-auxiliary-lower} after increasing $C$ and
$n_0$ if necessary.

Taking the maximum over all $n$-vertex members of
$\cP(G;m,t)$ in \eqref{eq:fixed-auxiliary-upper}, using the member
provided by \eqref{eq:fixed-auxiliary-lower}, and then applying
\eqref{eq:family-density}, we obtain
\eqref{eq:fixed-auxiliary-density}.
\end{proof}

\subsection{\texorpdfstring{Families of auxiliary $\ell$-graphs}%
{Families of auxiliary ell-graphs}}
\label{sec:auxiliary-hypergraphs}
The preceding subsection determines the density of $\cP(G;m,t)$ when
$G$ is fixed.  We now allow the auxiliary graph to range over a
family.  This requires some care: an $n$-vertex member of the resulting
union may be witnessed by an auxiliary graph that depends on both the
member and $n$, so the fixed-$G$ density formula alone does not give an
upper bound for the union.

\begin{definition}\label{def:auxiliary-family}
For a nonempty family $\G$ of $\ell$-graphs, define
\[
 \cP(\G;m,t):=\bigcup_{G\in\G}\cP(G;m,t).
\]
\end{definition}

The family $\cP(\G;m,t)$ is monotone, being a union of monotone
families.  Moreover, every
member of $\cP(\G;m,t)$ is $K_{m(k-1)+1}^{(k)}$-free.  Consequently, $\F\bigl(\cP(\G;m,t)\bigr)\ne\varnothing$.

The next lemma shows that the density of $\cP(\G;m,t)$ is determined by
the supremum of the Lagrangians of the auxiliary graphs.

\begin{lemma}\label{lem:supremum}
	Let $k>\ell\ge2$, and let $m,t$ be integers satisfying
	$m\ge k$ and $k-\ell\le t\le m-\ell$.  Then, for every nonempty family
	$\G$ of $\ell$-graphs,
	\[
	\gammakl\bigl(\F(\cP(\G;m,t))\bigr)
	=c_{m,t}\sup_{G\in\G}\lambda(G).
	\]
\end{lemma}

\begin{proof}
	Set $\tau:=\sup_{G\in\G}\lambda(G)$ and $\cP:=\cP(\G;m,t)$.
	Let $C$ and $n_0$ be the constants supplied by
	\cref{thm:fixed-auxiliary}.  
	For every $n\ge n_0$ and every $n$-vertex $H\in\cP$, there is some
	$G\in\G$ such that $H\in\cP(G;m,t)$.  The upper bound in
	\cref{thm:fixed-auxiliary} therefore gives
	\[
	\frac{\delta_\ell(H)}{\binom{n-\ell}{k-\ell}}
	\le c_{m,t}\lambda(G)+Cn^{-1/2}
	\le c_{m,t}\tau+Cn^{-1/2}.
	\]
	Taking the maximum over all such $H$ and using
	\eqref{eq:family-density}, we obtain
	\[
	\gammakl(\F(\cP))\le c_{m,t}\tau.
	\]
	
	For the reverse inequality, fix $\eps>0$ and choose $G\in\G$ such that
	$\lambda(G)>\tau-\eps$.  By the lower bound in
	\cref{thm:fixed-auxiliary}, for every $n\ge n_0$ there exists an
	$n$-vertex $H\in\cP(G;m,t)\subseteq\cP$ satisfying
	\[
	\frac{\delta_\ell(H)}{\binom{n-\ell}{k-\ell}}
	\ge c_{m,t}(\tau-\eps)
	-C\left(n^{-1}+n^{-(k-\ell)/2}\sqrt{\log n}\right).
	\]
	Applying \eqref{eq:family-density} and letting $n\to\infty$ gives
	\[
	\gammakl(\F(\cP))\ge c_{m,t}(\tau-\eps).
	\]
	Since $\eps>0$ was arbitrary, the reverse inequality follows.
\end{proof}

Note that the construction and all the arguments in this
section remain valid when $\ell=2$.  The stronger assumption
$\ell>2$ in the main theorem is used only in the next section, where
\cref{lem:one-sided} is applied to auxiliary $\ell$-graphs; that input
is available only for uniformity at least $3$.

\section{Proof of the main theorem}\label{sec:main-proof}
Throughout this section, fix integers $k>\ell>2$.  We construct, for
an arbitrary $\alpha\in[0,1)$, a nonempty monotone family $\cQ$ such
that $\gammakl(\F(\cQ))=\alpha$,
while every finite subfamily of $\F(\cQ)$ has strictly larger
$\ell$-degree Tur\'an density.  The constructions for positive density
and for zero density are different.

\subsection{Positive densities}\label{sec:positive}
We first treat the case $\alpha>0$.  The proof uses the construction
from \cref{sec:construction} to turn Lagrangian thresholds for auxiliary
$\ell$-graphs into a nested collection of monotone families of
$k$-graphs whose $\ell$-degree Tur\'an densities can be computed
explicitly.  We then choose the threshold that gives density $\alpha$
and use the nesting of these families to show that every finite
subfamily of the resulting forbidden family has strictly larger
density.

\begin{lemma}\label{lem:positive}
Let $k>\ell>2$ and $0<\alpha<1$.  There exists a nonempty monotone
family $\cQ$ of $k$-graphs such that $\gammakl(\F(\cQ))=\alpha$
and $\gammakl(\F')>\alpha$ for every finite subfamily
$\F'\subseteq\F(\cQ)$.
\end{lemma}

\begin{proof}
	Let $b_\ell\in(0,1)$ be the constant given by
	\cref{lem:one-sided}.  We first choose the parameters $m,t$ so that
	the coefficient $c_{m,t}$ is slightly larger than $\alpha$, but strictly smaller than $\alpha/b_\ell$.  This will ensure that the threshold
	$\tau:=\alpha/c_{m,t}$ lies in the interval $(b_\ell,1)$, where
	\cref{lem:one-sided} applies.
	
	Since
	$0<\alpha<1$, choose $\rho \in(0,1)$ such that
	$\alpha< \rho^{k-\ell}<\min\{1,\alpha/b_\ell\}$.  For each sufficiently large
	integer $m$, put $t:=\lfloor \rho m\rfloor$.  Then
	$k-\ell\le t\le m-\ell$, and
	\[
	c_{m,t}
	=\prod_{j=0}^{k-\ell-1}\frac{\lfloor \rho m\rfloor-j}{m}
	\longrightarrow \rho^{k-\ell}
	\qquad\text{as }m\to\infty.
	\]
	Thus, we can choose $m\ge k$ and $t=\lfloor \rho m\rfloor$ such that $\alpha<c_{m,t}<\min\{1,\alpha/b_\ell\}$.
	Set $\tau:=\alpha/c_{m,t}$.  Then $b_\ell<\tau<1$.
	
	For each $u\in(b_\ell,1)$, define
	\[
	\G(u):=
	\{G:G\text{ is an $\ell$-graph and }\lambda(G)<u\} \qquad {\rm and} \qquad \cQ(u):=\cP(\G(u);m,t).
	\]
	By \cref{lem:one-sided}, we have $\sup_{G\in\G(u)}\lambda(G)=u$,
	and hence \cref{lem:supremum} gives
	\begin{equation}\label{eq:threshold-density}
		\gammakl\bigl(\F(\cQ(u))\bigr)=c_{m,t}u.
	\end{equation}
	Moreover, $\cQ(u)$ is monotone, and these families are nested:
	if $u<v$, then
	$\cQ(u)\subseteq\cQ(v)$.
	
	For $i\ge1$, let
	\[
	\tau_i:=\tau+\frac{1-\tau}{i+1},
	\qquad
	\cQ:=\bigcap_{i\ge1}\cQ(\tau_i).
	\]
	Then we have $\tau_i\downarrow\tau$ and $	\cQ(\tau)\subseteq\cQ\subseteq\cQ(\tau_i)$ for every $i\ge1$.
	In particular, $\cQ$ is monotone and contains $k$-graphs of arbitrarily
	large order.
	By \eqref{eq:threshold-density} and the monotonicity of the density
	with respect to the underlying monotone family, we have
	\[
	\alpha
	=c_{m,t}\tau
	=\gammakl\bigl(\F(\cQ(\tau))\bigr)
	\le \gammakl\bigl(\F(\cQ)\bigr)
	\le \gammakl\bigl(\F(\cQ(\tau_i))\bigr)
	=c_{m,t}\tau_i
	\]
	for every $i\ge1$.  Letting $i\to\infty$, we obtain
	\[
	\gammakl\bigl(\F(\cQ)\bigr)=\alpha.
	\]

	Since the sequence $\cQ(\tau_i)$ is decreasing, the sequence
	$\F(\cQ(\tau_i))$ is increasing.  Moreover,
	\begin{equation}\label{eq:complement-union}
		\F(\cQ)
		=\F\left(\bigcap_{i\ge1}\cQ(\tau_i)\right)
		=\bigcup_{i\ge1}\F(\cQ(\tau_i)).
	\end{equation}
	
	Now let $\F'\subseteq\F(\cQ)$ be a nonempty finite family.  By
	\eqref{eq:complement-union} and the fact that the families $\F(\cQ(\tau_i))$ are
	increasing, there is an index $i$ such that $\F'\subseteq \F(\cQ(\tau_i))$.
	Hence, we have
 	\[
	\gammakl(\F')
	\ge \gammakl(\F(\cQ(\tau_i)))
	=c_{m,t}\tau_i
	>c_{m,t}\tau
	=\alpha.
	\]
	This
	completes the proof.
\end{proof}

\subsection{Zero density}\label{sec:zero}
We now turn to the case $\alpha=0$.  We use a sequence of simpler
cyclic multipartite families whose $\ell$-degree Tur\'an densities
tend to zero as the number of parts increases.  To ensure that no
finite subfamily of the eventual forbidden family already has zero
density, we also show that finitely many of these
constructions can be realized simultaneously, after passing to a
suitable refinement of their underlying partitions.  These two
features will yield a monotone family of density zero for which every
finite subfamily of the full forbidden family has positive density.

\begin{definition}\label{def:uncolored-construction}
Fix $k>\ell>2$.  Let $m\ge k$ and let $\V=(V_0,\ldots,V_{m-1})$ be an ordered
partition of a vertex set $V$ into $m$ possibly empty parts, with
indices taken modulo $m$.  Define $J_m(\V)$ to be the $k$-graph on
$V$ whose edges are the following:
\begin{enumerate}[label=\textup{(\roman*)},leftmargin=*,itemsep=3pt]
\item all $k$-sets meeting every part in at most $\ell-1$ vertices;
\item all sets $S\mathbin{\dot\cup}T$ such that
      $S\in\binom{V_i}\ell$ for some $i\in\mathbb Z_m$ and
      $T\in\binom{V}{k-\ell}$ satisfies
      $|T\cap V_{i+h}|=1$ for every $h\in[k-\ell]$.
\end{enumerate}
Let $\cR_m$ be the family of all subgraphs of
$J_m(\V)$, where
$\V$ ranges over all ordered $m$-partitions.
\end{definition}

The family $\cR_m$ is monotone by the same argument as for
$\cP(G;m,t)$.  Each part of a
witnessing partition is independent, so every member of $\cR_m$ is
$K_{m(k-1)+1}^{(k)}$-free.  In particular,
$\F(\cR_m)\ne\varnothing$.

The family $\cR_m$ is related to the construction from \cref{sec:construction}
by
\begin{equation}\label{eq:uncolored-union}
	\cR_m=\bigcup_{a\ge\ell}\cP(K_a^{(\ell)};m,k-\ell).
\end{equation}
Indeed, every cyclic multipartite $K_a^{(\ell)}$-construction with
$t=k-\ell$ is a subgraph of the corresponding graph $J_m(\V)$.  Conversely, given a partition
$\V$, choose $a\ge\max\{\ell,|V_0|,\ldots,|V_{m-1}|\}$. Then we choose an injection
$\phi_i:V_i\to V(K_a^{(\ell)})$ for each $i \in \mathbb Z_m$ and set
$\phi_{i,T}:=\phi_i$ for every admissible extension set $T$.  Thus we have
\[
H(K_a^{(\ell)};\V,\Phi)=J_m(\V),
\]
which proves \eqref{eq:uncolored-union}.

Assigning weight $1/a$ to every vertex of $K_a^{(\ell)}$ gives
\[
\lambda(K_a^{(\ell)})
\ge \frac{\ell!\binom{a}{\ell}}{a^\ell}
\longrightarrow1
\qquad\text{as }a\to\infty.
\]
Since every Lagrangian is at most $1$, it follows that $\sup_{a\ge\ell}\lambda(K_a^{(\ell)})=1$.
Combining this with \eqref{eq:uncolored-union} and
\cref{lem:supremum}, we obtain
\begin{equation}\label{eq:uncolored-density}
	\gammakl(\F(\cR_m))
	=c_{m,k-\ell}
	=\frac{(k-\ell)!}{m^{k-\ell}}.
\end{equation}

We next show that the intersection of finitely many families
$\cR_m$ still has positive density.

\begin{lemma}\label{lem:intersection}
Let $q\ge1$ and $m_1,\ldots,m_q\ge k$.  Then
\begin{equation}
\label{eq:intersection-lower}
 \gammakl\left(\F\left(\bigcap_{j=1}^q\cR_{m_j}\right)\right)
 \ge\prod_{j=1}^q\frac{(k-\ell)!}{m_j^{k-\ell}}>0.
\end{equation}
\end{lemma}

\begin{proof}
For every sufficiently large $n$, let $V$ be an $n$-element vertex set, and put $M:=\prod_{j=1}^q m_j$. 
For each
$\boldsymbol a=(a_1,\ldots,a_q)\in
\prod_{j=1}^q\{0,\ldots,m_j-1\}$,
choose a set $W_{\boldsymbol a}$ so that $V=\mathbin{\dot\bigcup}_{\boldsymbol a}W_{\boldsymbol a}$
and the sizes of the $M$ sets $W_{\boldsymbol a}$ differ by at most
one.
For each $j\in[q]$, the $j$-th coordinate defines an ordered
$m_j$-partition $\V^{(j)}$ by
\[
 V_i^{(j)}:=
 \bigcup_{\boldsymbol a:\,a_j=i}W_{\boldsymbol a}
 \qquad(0\le i<m_j).
\]
Let
$J_j:=J_{m_j}(\V^{(j)})$, and define $H$ by $ E(H):=\bigcap_{j=1}^q E(J_j)$.
Then $H\in\bigcap_{j=1}^q\cR_{m_j}$. For $v\in V$, let $\boldsymbol a(v)
=(a_1(v),\ldots,a_q(v))$
be the index of the unique part \(W_{\mathbf a}\) containing $v$.  Thus
$v\in V_{a_j(v)}^{(j)}$ for every $j\in[q]$.

Fix $S\in\binom V\ell$.  For $j\in[q]$, let $h_j$ be the number of
parts of $\V^{(j)}$ met by $S$.
We now count ordered $(k-\ell)$-tuples
$(v_1,\ldots,v_{k-\ell})$ of distinct extension vertices.  For each
$j\in[q]$, we impose the following conditions on coordinate $j$.
If $h_j=1$, say $S\subseteq V_i^{(j)}$, require the extension
vertices to occupy
$V_{i+1}^{(j)},\ldots,V_{i+k-\ell}^{(j)}$, one in each part.
There are $(k-\ell)!$ assignments of these part indices to the ordered
vertices.  If
$h_j\ge2$, require their part indices to be distinct and to avoid all
parts met by $S$; there are
$(k-\ell)!\binom{m_j-h_j}{k-\ell}$ assignments.  Accordingly, set
\[
 A_j(S):=
 \begin{cases}
 (k-\ell)!,& h_j=1,\\[1mm]
 (k-\ell)!\binom{m_j-h_j}{k-\ell},& h_j\ge2.
 \end{cases}
\]
Since $h_j\le\ell$ and $m_j\ge k$, we have
$m_j-h_j\ge k-\ell$ and hence $A_j(S)\ge(k-\ell)!$ in both cases.

There are $\prod_{j=1}^q A_j(S)$ choices for the coordinate
assignments.  Each choice specifies $k-\ell$ pairwise distinct parts
$W_{\boldsymbol a}$, all disjoint from $S$.  Hence it gives $\left(\frac nM\right)^{k-\ell}
-O(n^{k-\ell-1})$
ordered lists of distinct vertices outside $S$.  Every resulting set
$S\cup\{v_1,\ldots,v_{k-\ell}\}$ is an edge of $J_j$ for each $j$: it is of
type~\textup{(ii)} when $h_j=1$ and of type~\textup{(i)} when
$h_j\ge2$.

The requirements are invariant under permutations of the extension
vertices, so every unordered extension is counted exactly $(k-\ell)!$
times.  Therefore
\[
 d_H(S) \ge\frac{n^{k-\ell}}{(k-\ell)!}
       \prod_{j=1}^q\frac{A_j(S)}{m_j^{k-\ell}}
       -O(n^{k-\ell-1}) \ge\frac{n^{k-\ell}}{(k-\ell)!}
       \prod_{j=1}^q\frac{(k-\ell)!}{m_j^{k-\ell}}
       -O(n^{k-\ell-1}).
\]
Dividing by
$\binom{n-\ell}{k-\ell}$ and letting $n\to\infty$ proves
\eqref{eq:intersection-lower}.
\end{proof}

The common refinement in \cref{lem:intersection} is essential: the
lower bound in \eqref{eq:intersection-lower} does not follow from the
individual densities of the families $\cR_{m_j}$ alone.

\begin{lemma}\label{lem:zero-density}
	There exists a nonempty monotone family $\cQ$ of $k$-graphs such
	that $\gammakl(\F(\cQ))=0$ and $\gammakl(\F')>0$
	for every finite subfamily $\F'\subseteq\F(\cQ)$.
\end{lemma}

\begin{proof}
	Set $\cQ:=\bigcap_{m\ge k}\cR_m$.
	This family is monotone and contains every edgeless $k$-graph, so it
	is nonempty.  Since $\cQ\subseteq\cR_m$ for every $m\ge k$,
	\eqref{eq:uncolored-density} and the monotonicity of the density give
	\[
	0\le\gammakl(\F(\cQ))
	\le\inf_{m\ge k}\gammakl(\F(\cR_m))
	=\inf_{m\ge k}\frac{(k-\ell)!}{m^{k-\ell}}
	=0.
	\]
	Thus $\gammakl(\F(\cQ))=0$.
	
	Let $\F'\subseteq\F(\cQ)$ be a nonempty finite family.  For every
	$F\in\F'$, choose $m(F)\ge k$ such that $F\notin\cR_{m(F)}$, and set $M:=\max_{F\in\F'}m(F)$.
	Every member of $\bigcap_{m=k}^M\cR_m$ is $\F'$-free.  Indeed, if
	$H\in\bigcap_{m=k}^M\cR_m$ contained some $F\in\F'$, then the
	monotonicity of $\cR_{m(F)}$ would imply
	$F\in\cR_{m(F)}$, a contradiction.  It follows from
	\cref{lem:intersection} that
	\[
		\gammakl(\F')
		\ge
		\gammakl\left(
		\F\left(\bigcap_{m=k}^M\cR_m\right)
		\right) \ge
		\prod_{m=k}^M\frac{(k-\ell)!}{m^{k-\ell}}
		>0.
	\]
	This completes
	the proof.
\end{proof}

We can now finish the proof of the main theorem.

\begin{proof}[Proof of \cref{thm:main}]
	Fix $\alpha\in[0,1)$.  Let $\cQ$ be the family supplied by
	\cref{lem:positive} if $\alpha>0$, and by \cref{lem:zero-density}
	if $\alpha=0$.  Set $\F:=\F(\cQ)$.
	Then $\gammakl(\F)=\alpha$, while $\gammakl(\F')>\alpha$ for every
	finite subfamily $\F'\subseteq\F$.
	In particular, $\F$ is infinite, since otherwise taking
	$\F'=\F$ would give a contradiction.
\end{proof}

\section{Conclusion}\label{sec:conclusion}
In this paper, we proved that, for every $k>\ell>2$ and
every $\alpha\in[0,1)$, there exists an infinite family $\F$ of
$k$-graphs with $\ell$-degree Tur\'an density $\alpha$, while every
finite subfamily of $\F$ has strictly larger density. In particular,
this determines the full spectrum $\Gamma_{\infty}^{k,\ell}=[0,1)$
in this range, which proves the conjecture of
Mubayi and Zhao on codegree Tur\'an densities for every $k>3$ and
answers the question of Lo and Markstr\"om in this range.
The proof combines a cyclic multipartite construction
with hypergraph Lagrangians to obtain nested monotone families
whose densities can be prescribed.

It remains natural to ask whether the stronger conclusion of
\cref{thm:main} also holds when $\ell=2$.  The cyclic multipartite
construction and the matching estimates in \cref{sec:construction}
remain valid in this setting.  The difficulty lies in the
positive-density argument, which uses a one-sided approximation of
Lagrangian values obtained from an interval of ordinary hypergraph
Tur\'an densities near $1$.  For graphs, the spectrum in
\eqref{eq:graph-spectrum} is discrete, so this input is no longer
available.  For $k=3$, the case $\ell=2$ is precisely the remaining
case of the conjecture of Mubayi and Zhao.  For $k\ge4$, Li and
Liu~\cite{Li-Liu-26} have determined the full spectrum for
$\ell=2$, but the stronger requirement that every finite subfamily
have strictly larger density remains open.

\section*{Declaration on the use of AI}
The authors used AI for language editing and exploratory discussion of proof presentation.  All mathematical statements, proofs, and references were independently written and finalized by the authors, who take full responsibility for the content of the paper.

\bibliographystyle{abbrv}
\bibliography{ref}

\end{document}